\documentclass[12pt]{amsart}

\usepackage{enumitem}
\usepackage{psfrag}
\usepackage{graphicx}
\usepackage{pinlabel}
\usepackage{graphicx}
\usepackage{amsmath}
\usepackage{amssymb}
\usepackage{amscd}
\usepackage{picinpar}
\usepackage{times}
\usepackage{pb-diagram}
\usepackage{graphicx}
\usepackage{wrapfig}
\usepackage{xspace}
\usepackage{hyperref}
\usepackage{url}
\usepackage{caption}
\usepackage{subcaption}
\newtheorem{theorem}{Theorem}[section]
\newtheorem{lemma}[theorem]{Lemma}
\newtheorem{proposition}[theorem]{Proposition}
\newtheorem{corollary}[theorem]{Corollary}
\newtheorem{definition}[theorem]{Definition}

\newtheorem{remark}[theorem]{Remark}
\newtheorem*{theorem*}{Theorem}
\begin{document}

\title{The convergence of discrete uniformizations for genus zero surfaces}

\author{Yanwen Luo}
\address{Department of Mathematics, Rutgers University, Piscataway, NJ, 08854}
\email{yl1594@math.rutgers.edu}

\author{Tianqi Wu}
\address{Center of Mathematical Sciences and Applications, Harvard University, Cambridge, MA 02138}
\email{tianqi@cmsa.fas.harvard.edu}

\author{Xiaoping Zhu}
\address{Department of Mathematics, Rutgers University, Piscataway, NJ, 08854}
\email{xz349@math.rutgers.edu}



\begin{abstract}
The notion of discrete conformality proposed by Luo \cite{luo2004combinatorial} and Bobenko-Pinkall-Springborn \cite{bobenko2015discrete} on triangle meshes has rich mathematical
theories and wide applications. Gu et al. \cite{gu2019convergence}\cite{wu2020convergence} proved that the discrete uniformizations approximate the continuous uniformization for closed surfaces of genus $\geq 1$, given that the approximating triangle meshes are reasonably good. In this paper, we generalize this result to the remaining case of genus zero surfaces, by reducing it to planar cases via stereographic projections. 

\end{abstract}
\maketitle
\tableofcontents
\section{Introduction}

           This work is a continuation of \cite{wu2020convergence}, which studies the convergence of discrete uniformizations to the continuous uniformization for closed surfaces of genus greater than zero.  The notion of discrete conformality and discrete uniformization discussed here, which is called vertex scaling, were 
           introduced by Luo \cite{luo2004combinatorial} and Bobenko-Pinkall-Springborn \cite{bobenko2015discrete}. Gu-Luo-Wu \cite{gu2019convergence} first proved the convergence of discrete uniformizations for the unit disk and tori.   
 
  In this paper, we prove the convergence result for surfaces of genus zero. Specifically, for a reasonable geodesic triangulation $T$ on a Riemannian surface $(M, g)$, the discrete uniformization of its induced polyhedral metric approximates the uniformization for $(M, g)$ by an error in the order of the maximal edge length of edges in $T$.
 There has been extensive study about the convergence of discrete conformality since Rodin-Sullivan's seminal work \cite{rodin1987convergence} on the convergence of circle packings. Other convergence results about the vertex scaling can also be found in \cite{luo2020discrete}\cite{bucking2018c}\cite{bucking2016approximation}.

\subsection{Set up and the main theorem}
Suppose $M$ is a compact orientable surface, possibly with boundary, and $T$ is a triangulation of $M$, which is always assumed to be a simplicial complex such that its one-skeleton is a $4$-vertex-connected graph. Denote $V(T),E(T),F(T)$ as the sets of vertices, edges, and triangles of $T$ respectively. Further denote $\text{int}(T)$ as the set of interior vertices of $T$, and $\text{bdy}(T)$ as the set of boundary vertices of $T$. If $M$ is equipped with a smooth Riemannian metric $g$, $T$ is a \textit{geodesic triangulation} if any edge in $T$ is a shortest geodesic arc in $(M,g)$.

An \emph{admissible edge length function} $l\in\mathbb{R}_{>0}^{E(T)}$ of $T$ satisfies the triangle inequality $l_{ij}+l_{jk}>l_{ik}$ for any triangle $\triangle ijk\in F(T)$. 
For a geodesic triangulation $T$ on a Riemannian surface $(M,g)$, we can naturally define an edge length function $l$ using the geodesic lengths of the edges. 
Given an admissible edge length function $l$, we can construct a piecewise Euclidean triangle mesh $(T,l)_E$ by isometrically gluing the Euclidean triangles with the edge lengths defined by $l$ along pairs of edges. Similarly, a spherical triangle mesh $(T,l)_S$ can be constructed by replacing Euclidean triangles with spherical triangles of the same edge lengths, provided that $l_{ij}+l_{jk}+l_{ki}<\pi$ for any triangle $\triangle ijk$ in $F(T)$. For given $(T,l)_E$ or $(T,l)_S$, we often use $\theta^i_{jk}$ to denote the inner angle at $i$ in the triangle $\triangle ijk$. We also define the discrete curvature $K_i$ at a vertex $i\in V(T)$ as 

$$ K_i = \left\{ 
       \begin{array}{ll}
         2\pi-\sum_{jk\in E:ijk\in F}\theta^i_{jk}, & \mbox{ if $i\in \text{int}(T)$},\\
        \pi-\sum_{jk\in E:ijk\in F}\theta^i_{jk}, & \mbox{ if $i\in \text{bdy}(T)$}.      
        \end{array}
\right. $$
A Euclidean (\textit{resp.} spherical) triangle mesh is globally flat (\textit{resp.} globally spherical) if and only if $K_i=0$ for any $i\in \text{int}(T)$.

\begin{definition}
Given a triangulation $T$, a \emph{discrete conformal factor} $u$ is a real-valued function on $V(T)$. 
For the Euclidean case, $(T,l)_E$ and $(T,l')_E$ are called \emph{discrete conformal} if for some discrete conformal factor $u$, 
\begin{equation}
\label{Euclidean discrete conformal}
l'_{ij}=e^{\frac{1}{2}(u_i+u_j)}l_{ij}
\end{equation}
for any $ij\in E(T)$. For the spherical case, $(T,l)_S$ and $(T,l')_S$ are \emph{discrete conformal} if for some discrete conformal factor $u$, 
\begin{equation}
\label{Spherical discrete conformal}
\sin\frac{l'_{ij}}{2}=e^{\frac{1}{2}(u_i+u_j)}\sin\frac{l_{ij}}{2}
\end{equation}
for any $ij\in E(T)$.
\end{definition}
We denote $l'=u*l$ if equation (\ref{Euclidean discrete conformal}) holds, and $l'=u*_s l$ if equation (\ref{Spherical discrete conformal}) holds. Given a triangle mesh $(T, l)_E$ (\textit{resp.} $(T, l)_S$), let $\theta^i_{jk}(u)$ and $K_i(u)$ denote the corresponding inner angle at $i$ in triangle $\triangle ijk$ and the discrete curvature at $i$ respectively in $(T, u*l)_E$ (\textit{resp.} $(T, u*_s l)_S$).
If $(T,l)_S$ is a topological sphere, $u$ is called a \emph{discrete uniformization factor} for $(T,l)_S$, if
$(T,u*_sl)_S$ is isometric to the unit sphere, which is equivalent to that the discrete curvature $K(u):=[K_i(u)]_{i\in V}$ is zero. A triangle mesh $(T,l)_E$ (resp. $(T,l)_S$) is called \emph{strictly Delaunay} if for any edge $ij$ in T, two adjacent triangles $\triangle ijk,\triangle ijk'$ containing $ij$ satisfies
\begin{equation}
\label{delaunay}
\theta^k_{ij}+\theta^{k'}_{ij}<
\theta^i_{jk}+\theta^i_{jk'}+
\theta^j_{ik}+\theta^j_{ik'}.
\end{equation}

We quantify the regularity of a triangle mesh as follows.



\begin{definition}
A triangle mesh $(T,l)_E$ (resp. $(T,l)_S$) is called \emph{$\epsilon$-regular} if
\begin{enumerate}
	\item[(a)] any inner angle $\theta^i_{jk}\geq\epsilon$, and 
	\item[(b)] for any adjacent triangles $\triangle ijk$ and $\triangle ijk'$, $\theta^k_{ij}+\theta^{k'}_{ij}\leq\pi-\epsilon$.
\end{enumerate}
\end{definition}
Condition (a) requires that any triangle is away from degenerating, and condition (b) requires the triangle mesh to be uniformly strictly Delaunay.
Let $|x| = \max_{i\in I} |x_i|$ denote the maximal norm of a vector $x\in \mathbb{R}^I$ in a finite dimensional vector space.  Let $\hat{\mathbb{C}}$ be the standard Riemann sphere, which can be identified with the unit sphere $\mathbb{S}^2$ in $\mathbb{R}^3$ by the stereographic projection. The main theorem of the paper is 
\begin{theorem}
\label{Main theorem}
Suppose $(M,g)$ is a closed smooth Riemannian surface of genus zero with three marked points $X, Y, Z$, and $\bar u\in C^{\infty}(M)$ is the unique uniformization conformal factor such that $(M,e^{2\bar u}g)$ is 
isometric to the unit sphere $\mathbb{S}^2 \cong \hat{\mathbb C}$ through map $\phi$, where $\phi(Z)=0$, $\phi(Y)=1$, $\phi(X)=\infty$. Assume $T$ is a geodesic triangulation of $(M,g)$, and $l\in\mathbb{R}_{>0}^{E(T)}$ denotes its edge length in $(M, g)$. 
Then for any $\epsilon>0$, there exists $\delta=\delta(M,g,X,Y,Z,\epsilon)>0$ such that if $(T,l)_S$ is $\epsilon$-regular and $|l|\leq\delta$, then
\begin{enumerate}
	\item[(a)] there exists a unique discrete conformal factor $u$ on $V(T)$, such that $(T,u*_sl)_S$ is strictly Delaunay and isometric to the unit sphere through a map $\psi$ such that $\psi(Z) = 0$, $\psi(Y) = 1$, and $\psi(X) = \infty$, and
	
	\item[(b)]  $|u-\bar u|_{V(T)}|\leq C|l|$ for some constant $C=C(M,g,X,Y,Z,\epsilon)>0$.
\end{enumerate}
\end{theorem}
\begin{remark}
\label{unique}
The uniqueness part of the theorem is already known as a consequence of Springborn's Theorem 10.5 in \cite{springborn2019ideal}, which is equivalent to Rivin’s earlier result on
hyperbolic polyhedral realization in \cite{rivin1994intrinsic}.
\end{remark}
\subsection{An equivalent formulation}
Springborn \cite{springborn2008conformal}\cite{springborn2019ideal} and Bobenko et al. \cite{bobenko2015discrete} proposed an another notion of discrete uniformization, which is for Euclidean triangle meshes that are homeomorphic to a sphere. We adapt their definitions as follows.

Let $\mathcal P$ be the set of the compact convex polyhedral surfaces $P$, satisfying that 

(a) $P$ is the boundary of the convex hull of a finite subset of $\mathbb S^2$, and 

(b) $0$ is strictly inside $P$, and

(c) each face of $P$ is a triangle. 

Given $P\in \mathcal P$, denote $V(P)$ as the set of its vertices, and $T_P$ as the natural triangulation of $P$ where each triangle is a face of $P$, and $l_P\in\mathbb R^{E(T_P)}$ as the edge length of $T_P$ on $P$. For a Euclidean triangle mesh $(T,l)_E$, which is a topological sphere, 
we say that $u$ is a \emph{discrete uniformization factor} of $(T,l)_{E}$ if $(T,u*l)_E$ is isometric to some $P\in\mathcal P$, through a map $\varphi$ such that $\varphi(T)=T_P$. 

We call a geodesic triangulation of the unit sphere \emph{strictly Delaunay} if the circumference circle of each triangle contains no other vertex. It is well-known that this spherical empty circle condition is equivalent to the condition that if for any edge $ij$ in T, two adjacent triangles $\triangle ijk,\triangle ijk'$ containing $ij$ satisfies
\begin{equation}
\theta^k_{ij}+\theta^{k'}_{ij}<
\theta^i_{jk}+\theta^i_{jk'}+
\theta^j_{ik}+\theta^j_{ik'}.
\end{equation}
If we consider this geodesic triangulation as a sphercial triangle mesh $(T,l)_{S}$ where 
$l(e)$ is  the geodesic arc length of edge $e$, then the above condition is just condition (\ref{delaunay}). \par 
The central projection $p:(x,y,z)\mapsto(x,y,z)/\sqrt{x^2+y^2+z^2}$ naturally gives rise to a bijection  between $\mathcal P$ and the set of strictly Delaunay triangulations of the unit sphere.
Here we always assume that a triangulation of $\mathbb S^2$ is a geodesic triangulation and each triangle is a proper subset of a hemisphere. For a vector $x\in\mathbb R^I$, we denote $\sin x$ as the vector in $\mathbb R^I$ such that
$(\sin x)_{i}=\sin (x_i)$.
\begin{proposition}
\label{equivalence}
Let $P\mapsto p(T_P)$ be a bijection from $\mathcal P$ to the set of strictly Delaunay triangulations of $\mathbb S^2$. Further we have that $l_P=2\sin\frac{l}{2}$ where $l$ denotes the geodesic edge lengths of $p(T_P)$ on $\mathbb S^2$.
\end{proposition}
\begin{proof}
Given $P\in\mathcal P$, for any triangle $\triangle ijk$, other vertices are on one side of the plane $\triangle ijk$ lies in by the convexity of $P$. Therefore other vertices lies outside the circumference circle of spherical triangle $p(\triangle ijk)$. So $p(T_P)$ is a strictly Delaunay triangulation of $\mathbb S^2$. See Figure \ref{Equivalence between Delaunay and convexity} for illustrations.

If $T$ is a strictly Delaunay triangulation of $\mathbb S^2$, we construct a polyhedral surface $P$ as the union of all flat triangles $\triangle ijk$ where $i,j,k$ are the three vertices of a triangle in $T$. Then $p|_P$ is a homeomorphism from $P$ to $\mathbb S^2$. Since $T$ satisfies the empty circle property, we conclude that the dihedral angle on any edge $ij\in E(T)$ is less than $\pi$ by the similar argument in the above paragraph. So $P$ is a convex polyhedral surface (See Lemma 6.1 in  \cite{devadoss2011discrete} for example). 

Since all the vertices of $P$ is on the unit sphere, therefore $P$ satisfies the condition $(a)$ of set $\mathcal{P}$. It is not hard to see that $P$ satisfies condition $(b)$ and $(c)$. The condition $l_P=2\sin\frac{l}{2}$ follows easily from the construction of $P$.
\end{proof}

\begin{figure}[h]
	 	\centering
	 \begin{subfigure}[h]{0.48\textwidth}
	 	 	\centering
	 	\includegraphics[width=0.81\textwidth]{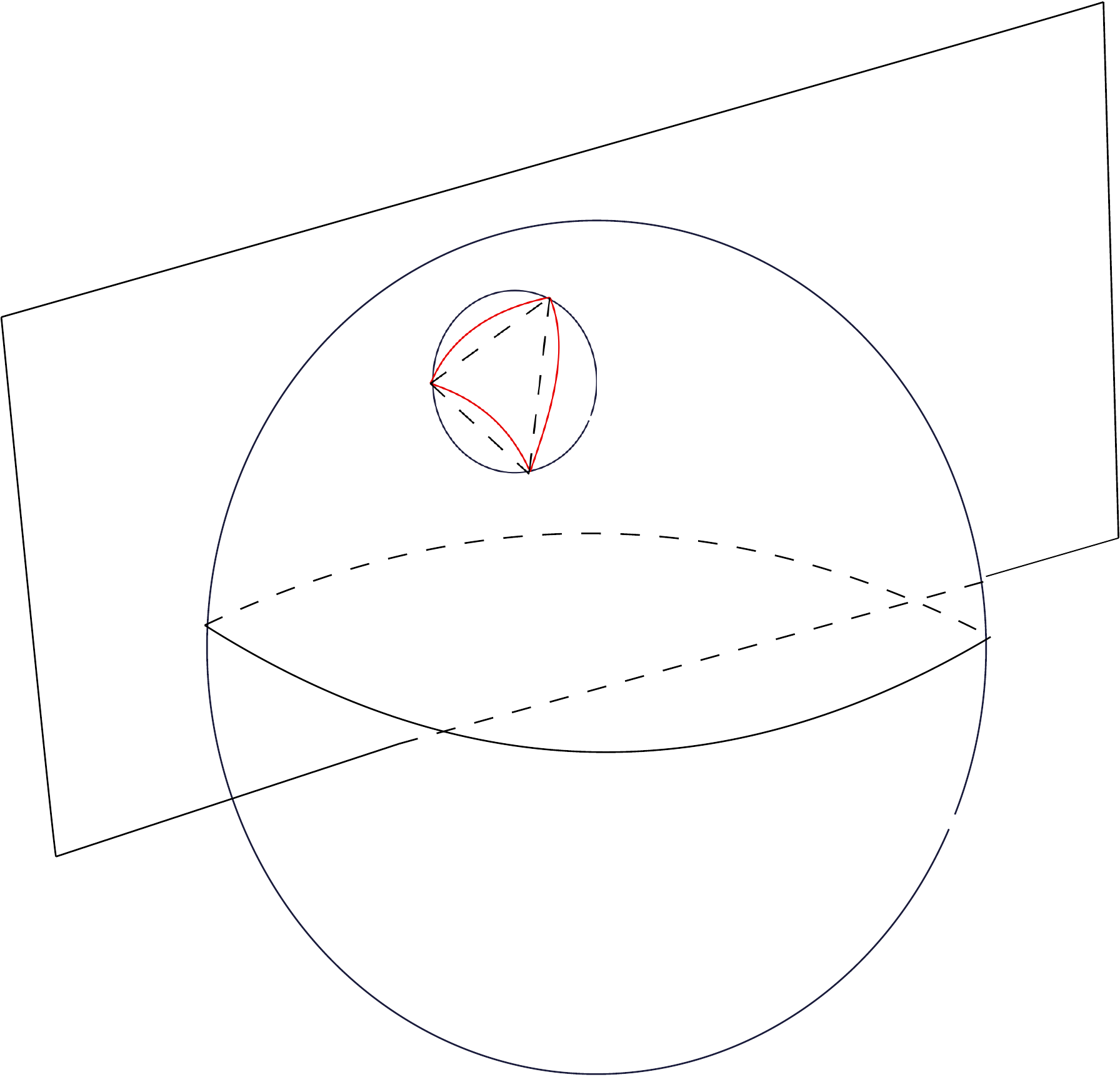}
	 \end{subfigure}
 \hspace{0.1cm}
    \begin{subfigure}[h]{0.45\textwidth}
	\centering
	\includegraphics[width=0.7\textwidth]{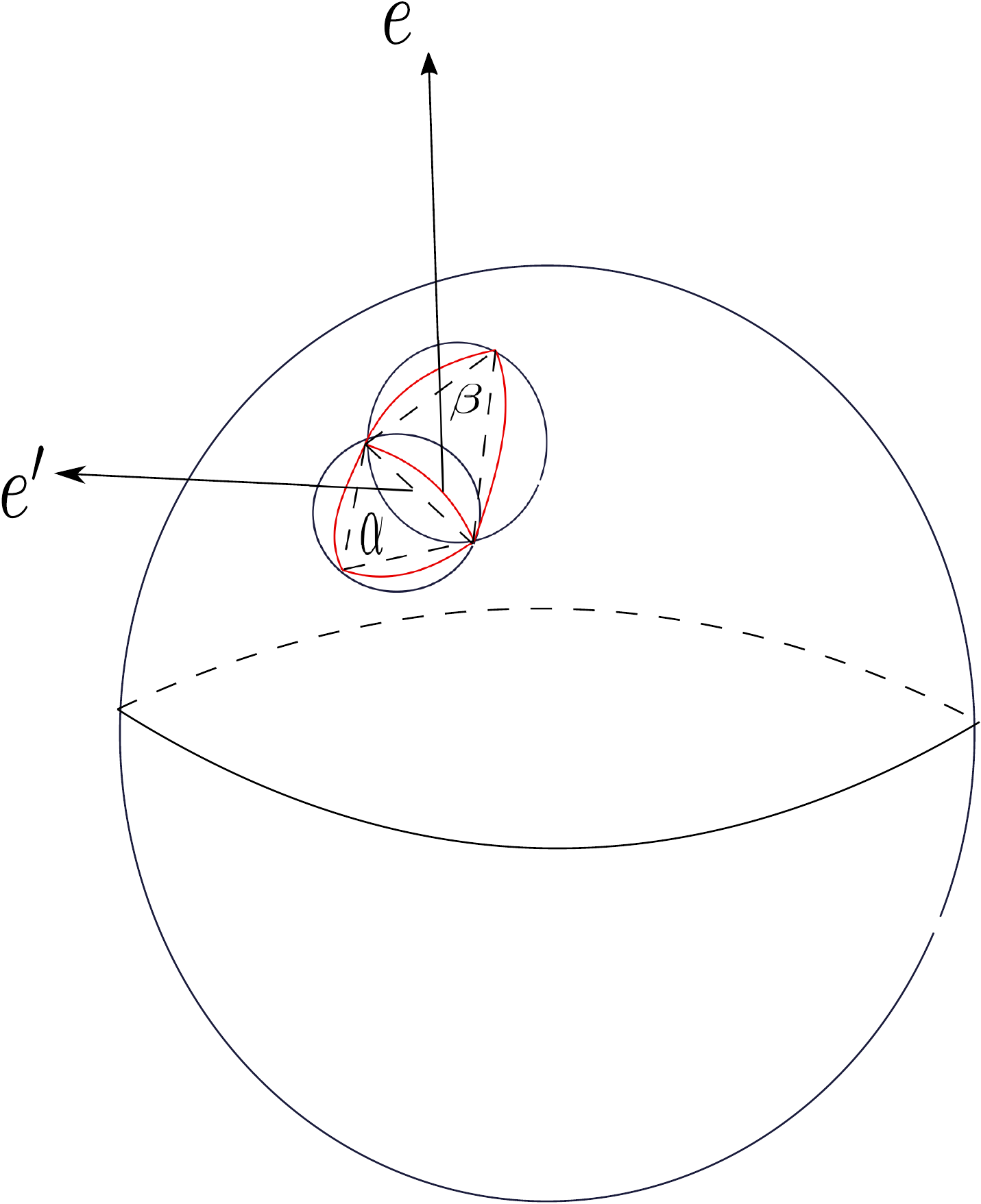}
    \end{subfigure}
    \caption{Equivalence between local Delaunay condition and local convexity.}
    \label{Equivalence between Delaunay and convexity}

\end{figure}

As a consequence, we obtain an equivalence between the two notions of discrete uniformizations. 
\begin{corollary}
Assume $T$ is topologically a sphere, then
    $u$ is a discrete uniformization factor of $(T,2\sin\frac{l}{2})_E$ if and only if $u$ is a discrete uniformization factor of $(T,l)_S$ and $(T,u*_sl)_S$ is strictly Delaunay.
\end{corollary}
By such an equivalence, we can reformulate our main theorem as follows.
\begin{theorem}
\label{main2}
Suppose $(M,g)$ is a closed smooth Riemannian surface of genus zero with three marked points $X, Y, Z$, and $\bar u\in C^{\infty}(M)$ is the unique uniformization conformal factor such that $(M,e^{2\bar u}g)$ is 
isometric to the unit sphere $\mathbb{S}^2 \cong \hat{\mathbb C}$ through map $\phi$, and $\phi(Z)=0$, $\phi(Y)=1$, $\phi(X)=\infty$. Assume $T$ is a geodesic triangulation of $(M,g)$, and $l\in\mathbb{R}_{>0}^{E(T)}$ denotes its edge length in $(M, g)$. 
Then for any $\epsilon>0$, there exists $\delta=\delta(M,g,X,Y,Z,\epsilon)>0$ such that if $(T,l)_S$ is $\epsilon$-regular and $|l|\leq\delta$, then
\begin{enumerate}
	\item[(a)] there exists a unique discrete conformal factor $u$ on $V(T)$, such that $(T,u*(2\sin\frac{l}{2}))_E$ is isometric to some $P\in\mathcal P$ through a map $\psi$ such that $\psi(Z) = 0$, $\psi(Y) = 1$, and $\psi(X) = \infty$, and
	
	\item[(b)]  $|u-\bar u|_{V(T)}|\leq C|l|$ for some constant $C=C(M,g,X,Y,Z,\epsilon)>0$.
\end{enumerate}
\end{theorem}
We will prove this new version of our main theorem. By the stereographic projection, we can consider triangulations of a flat polygon, instead of the polyhedrons inscribed in the unit sphere.
To obtain a satisfactory flat triangle mesh, we adapt the idea in \cite{wu2020convergence} and construct a discrete curvature flow on the triangle mesh. The estimate in part (b) essentially follows from a discrete elliptic estimate on the flow.

\subsection{Organization of the paper}
In Section 2, we will introduce the discrete calculus on graphs and the key discrete elliptic estimate. In Section 3 we will discuss the stereographic projection and the one-to-one correspondence between the convex polyhedral surfaces inscribed in the unit sphere and the the Delaunay triangulations of convex polygons. In Section 4, we will introduce some useful and elementary estimates about triangulations.
The proof of the main theorem \ref{main2} will be given in Section 5. 
\subsection{Acknowledgement} The work is supported in part by NSF 1719582,  NSF 1760471, NSF 1760527, NSF DMS 1737876, and NSF 1811878.

\section{Discrete Calculus on Graphs}

Assume $G=(V,E)$ is an undirected connected simple graph, on which we will frequently consider vectors in
$\mathbb R^V$, $\mathbb R^E$ and $\mathbb R^E_A$. 
Here $\mathbb R^E$ and $\mathbb R^E_A$ are both vector spaces of dimension $|E|$ such that
\begin{enumerate}
	\item[(a)] a vector $x$ in $\mathbb R^E$ is represented symmetrically, i.e., $x_{ij}=x_{ji}$, and
	\item[(b)] a vector $x$ in $\mathbb R^E_A$ is represented anti-symmetrically, i.e., $x_{ij}=-x_{ji}$.
\end{enumerate}
A vector $x$ in $\mathbb R^E_A$ is also called a \emph{flow} on $G$. An \emph{edge weight} $\eta$ on $G$ is a vector in $\mathbb R^E$. Given an edge weight $\eta$ and a vector $x\in\mathbb R^V$, its \emph{gradient} $\nabla x=\nabla_\eta x$ is a flow in $\mathbb R^E_A$ defined as 
$$
(\nabla x)_{ij}=\eta_{ij}(x_j-x_i).
$$
Given a flow $x\in\mathbb R^E_A$, its \emph{divergence} $div(x)$ is a vector in $\mathbb R^V$ defined as 
$$
div(x)_i=\sum_{j\sim i}x_{ij}.
$$
Given an edge weight $\eta$, the associated \emph{Laplacian} $\Delta = \Delta_\eta : \mathbb{R}^V \to \mathbb{R}^V$ is defined as $\Delta x=\Delta_\eta x=div(\nabla_\eta x)$, i.e., 
$$
(\Delta x)_i=\sum_{j\sim i}(\nabla x)_{ij}=\sum_{j\sim i}\eta_{ij}(x_j-x_i).
$$
The Laplacian is a linear transformation on $\mathbb R^V$, and could be identified as a symmetric $|V|\times |V|$ matrix. 

\subsection{Isoperimetric conditions and a discrete elliptic estimate}
Analogous to the isoperimetric condition on a Riemannian surface, we introduce the notion of $C$-isoperimetric conditions for a graph $G=(V,E)$ with a positive vector $l\in \mathbb{R}^{E}_{>0}$. 
Given $V_0\subset V$, we denote 
$$
\partial V_0=\{ij\in E:i\in V_0,\text{ and }j\notin V_0\}
$$ 
(See Figure \ref{isoperimetric figure} below),
and define the \textit{l-perimeter} of $V_0$ and the \textit{l-area} of $V_0$ as 
$$
|\partial V_0|_l = \sum_{ij\in \partial V_0} l_{ij}, \quad
\text{and}\quad|V_0|_l=\sum_{ij\in E:i,j\in V_0}l_{ij}^2.
$$
Given a constant $C>0$, such a pair $(G,l)$ is called $C$-\emph{isoperimetric} if for any $V_0\subset V$,
$$
\min\{|V_0|_l, |V|_l - |V_0|_l \}\leq C|\partial V_0|_l^2.
$$
\begin{figure}[h]
\centering
\includegraphics[width=0.6\textwidth]{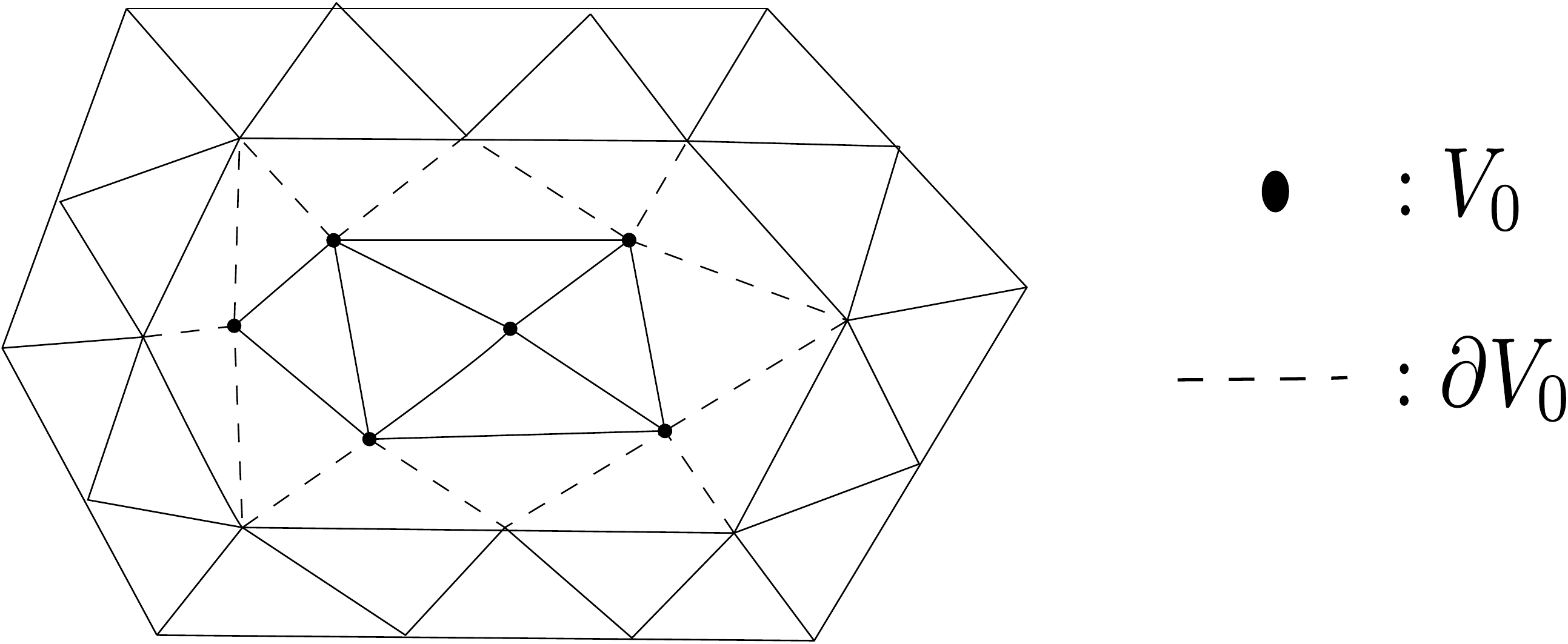}
\caption{The $C$-isoperimetric condition on graphs.}
\label{isoperimetric figure}
\end{figure}

We will see in Section 4 that a uniform $C$-\emph{isoperimetric} condition is satisfied by $\epsilon-$ regular triangle meshes approximating a closed smooth surface. The following discrete elliptic estimate is a key tool to prove our theorem of convergence, and is reformulated from the second conclusion of Lemma 2.3 in \cite{wu2020convergence}. 

\begin{lemma}
	\label{estimate for discrete elliptic operator}
	Given a constant $C>0$ and a $C$-isoperimetric pair $(G,l)$, consider the equation 
	\begin{equation}
	\label{discrete elliptic equation}
	(D-\Delta_\eta)u=div(x)+y
	\end{equation}
	on $G$ where
	\begin{enumerate}
	    \item [(i)] $\eta\in\mathbb R^E$ is an edge weight such that
	     for any $ij\in E$ 
	     $$
	     \eta_{ij}\geq \frac{1}{C},
	     $$ 
	     and
	    
	    \item [(ii)] $x\in\mathbb R^E_A$ is a flow such that for any $ij\in E$ 
	    $$
	    |x_{ij}|\leq Cl_{ij}^2,
	    $$ 
	    and
	    
	    \item [(iii)] $D\in\mathbb R^{V\times V}$ is a nonzero nonnegative diagonal matrix, and 
	    
	    \item [(iv)] $y\in\mathbb R^V$ satisfies that for any $i\in V$
	    $$
	    |y_i|\leq C \cdot D_{ii}|l|\cdot|V|_l^{\frac{1}{2}}.
	    $$
	\end{enumerate}
Then the solution $u\in\mathbb R^V$ of equation (\ref{discrete elliptic equation}) satisfies that
\begin{equation*}
	|u|\leq C'|l|\cdot|V|_{l}^{\frac{1}{2}}.
\end{equation*}
for some constant $C'=C'(C)>0$.
\end{lemma}

 \subsection{The differential of the discrete curvature}
There is an explicit formula for the infinitesimal changes of the
discrete curvature, as the mesh deforms in its discrete conformal class. 
\begin{proposition}
[Proposition 4.1.6 in \cite{bobenko2015discrete}]
\label{curv diff}
	Given a triangle mesh $(T,l)_E$, if $u*l$ is an admissible edge length, denote $\theta^i_{jk}(u)$ as the inner angle at $i$ in triangle $\triangle ijk$ in $(T,u*l)_E$, and  $K(u)\in\mathbb R^V$ as the discrete curvature in $(T,u*l)_E$. We also define the cotangent weight $\eta\in\mathbb R^E$ as 
	$$
	\eta_{ij}(u)=\frac{1}{2}\cot\theta^k_{ij}(u)+\frac{1}{2}\cot\theta^{k'}_{ij}(u)
	$$
	if $\triangle ijk,\triangle ijk'\in F(T)$ sharing edge $ij$, and 
	$$
\eta_{ij}(u)=\frac{1}{2}\cot\theta_{ij}^k	
	$$
	if $\triangle ijk\in F(T)$ and $ij$ is a boundary edge of $T$.
	Then we have
	$$
	\frac{\partial K}{\partial u}(u)=-\Delta_{\eta(u)}.
	$$
\end{proposition}

\section{Stereographic Projections of Triangle Meshes}
We will use the stereographic projection to connect convex polyhedral surfaces inscribed in the unit spheres with triangulations of planar convex polygons.
Denote $N$ as the north pole $(0,0,1)$ of the unit sphere $\mathbb S^2$.
The stereographic projection $p_N$ is a map from $\mathbb R^3\backslash\{z = 1\}$ to the $xy$-plane, which is identified as $\mathbb R^2$ or $\mathbb C$. The map $p_N$ is defined as 
$$
p_N(x,y,z)=\frac{x}{1-z}+i\frac{y}{1-z}.
$$
It is well known that the restriction of $p_N$ on $\mathbb S^2\backslash\{N\}$ is a conformal diffeomorphism to $\mathbb R^2$, and maps any circle to a circle or a straight line.
Given a convex polyhedral surface $P$ in $\mathcal P$ such that $N$ is a vertex of $P$, denote $\mathring T_P$ as the subtriangulation of $T_P$ with the open 1-star neighborhood of $N$ in $T_P$ being removed.
In this case, $\mathring P$ denotes the carrier of $\mathring T_P$ and is a topological closed disk.
We use $|\cdot|_2$ to denote the standard $l^2$-norm.

 \begin{lemma}
 \label{injectivity}
 (a) Assume $P\in\mathcal P$ and $P$ contains $N$ as a vertex, then $p_N$ is injective on $\mathring{P}$, and $Q=p_N(\mathring P)$ is a convex polygon, and $T_Q=p_N(\mathring T_P)$ is a geodesic triangulation of $Q$. 
Further, if we naturally identify $\mathring T_P$ and $T_Q$, and denote $l_P$ (resp. $l_Q$) as its edge length on $P$ (resp. $Q$), then
 	\begin{equation}
 	\label{discrete factor for projection}
	l_Q=w*l_P\text\quad{where}\quad
	w_{i}=\log\frac{2}{|i-N|_2^2}=\log\frac{|p_{N}(i)|_2^2+1}{2}, \quad \forall i\in V(\mathring T_P).
	\end{equation}
 (b) Assume $Q$ is a convex polygon in $\mathbb R^2$, and $T_Q$ is a strictly Delaunay triangulation of $Q$ such that $0$ is an interior vertex and $K_i>0$ for any boundary vertex $i$ in $V(T_Q)$. Then there exists a convex polyhedral surface $P\in\mathcal P$ such that $N\in P$ and $p_N(\mathring P)=Q$ and $p_N(\mathring T_P)=T_Q$.
 \end{lemma}
 \begin{proof}
(a) Let us first prove the injectivity by contradiction. Suppose $x,y$ are two different points on $\mathring{P}$ such that $p_N(x)=p_N(y)$. Then $N,x,y$ are co-linear and pairwise different. Without loss of generality, assume $y$ is between $x$ and $N$. Then it is not difficult to show that the line segment $\overline{xN}\subset P$.
 	Suppose $\triangle ijk$ is a face of $P$ containing the line segment $\overline{xN}$. Then $N$ has to be one of the vertex of $\triangle ijk$, and $x$, $y$ are contained in the edge in $\triangle ijk$ opposite to $N$. But this implies that $N, x, y$ are not co-linear, which leads to a contradiction.
 	
 	So $Q$ is a polygon, and $T_Q$ is a geodesic triangulation of $Q$. Any inner angle of the polygon $Q$ is less than $\pi$, since the dihedral angle on any edge $Ni\in E(T_P)$ is less than $\pi$. Equation (\ref{discrete factor for projection}) can be proved by a standard computation.
 
 (b) We will first construct a polyhedron $P$ and then show that it is satisfactory. The set of vertices of $P$ is given by
  $V_P=(p_N|_{\mathbb S^2})^{-1}(V(T_Q))\cup\{N\}$, and the set of faces of $P$ is given by a set of flat triangles in $\mathbb R^3$ with vertices in $V_P$. The $\triangle ijk$ is a triangle in $P$ if and only if $p_N(\triangle ijk)$ is a triangle in $T_Q$, or $\{i,j,k\}=\{N,x,y\}$ where $p_N(xy)$ is a boundary edge of $T_Q$. 
  
  It is ordinary to verify that the union $P$ of such triangles is a topological sphere, and these flat triangles naturally give a triangulation $T_P$ of $P$. It remains to show that $P\in\mathcal P$, or indeed that any dihedral angle in $T_P$ is less than $\pi$.
  
  Assume $ij$ is an edge in $T_P$. If $i=N$, the dihedral angle at $ij$ is less than $\pi$ because the discrete curvature at $p_N(j)$ in $T_Q$ is greater than $0$. If $p_N(ij)$ is a boundary edge in $T_Q$, assume $\triangle ijk\in F(T_P)$ and $k\neq N$, and then the dihedral angle at $ij$ is less than $\pi$ because $p_N(k)$ and $0$ are in the same half plane divided by $p_N(ij)$. Now we can assume that $p_N(ij)$ is an inner edge in $T_Q$, and $p_N(\triangle ijk), p_N(\triangle ijk')$ are two triangles in $T_Q$. Since $T_Q$ is strictly Delaunay, $p_N(k')$ is strictly outside of the circumcircle of $p_N(\triangle ijk)$. Since $p_N|_{\mathbb S^2}$ preserves circles, $k'$ is strictly outside of the spherical circumcircle of $\{i,j,k\}$ on $\mathbb S^2$. So the dihedral angle at $ij$ is less than $\pi$.
\end{proof}

In the following lemma we prove that the stereogarphic projection preserves the $\epsilon$-regularity.
\begin{lemma}
	\label{regular for stereographic condition}
	Assume $P\in\mathcal P$, and $N\in P$, and $T=p( T_P)$ is a geodeisc triangulation of $\mathbb S^2$,  and $Q=p_N(\mathring P)$, and $T_Q=p_N(\mathring T_P)$, 
	and $l$ (resp. $l_Q$) denotes the edge length of $T$ (resp. $T_Q$) on $\mathbb S^2$ (resp. $Q$).
	Then for any $\epsilon>0$, there exists constants $\epsilon'=\epsilon'(\epsilon)>0$ and $\delta=\delta(\epsilon)>0$ such that if $(T,l)_S$ is $\epsilon$-regular and $|l|<\delta$, then $(T_Q,l_Q)_E$ is $\epsilon'$-regular, and $K_i\geq\epsilon'$ for any boundary vertex $i$ in $T_Q$.
\end{lemma}
\begin{proof}
	Let $\theta_{ij}^k$ denote the inner angles in $(T,l)_S$, and $\phi_{ij}^k$ denote the inner angles in $(T_Q,l_Q)_E$. We need to prove following three statements: (a) $\phi_{ij}^k$ are bounded below by $\epsilon'>0$, and (b) $T_Q$ is strictly Delaunay with angle sums $\phi_{ij}^k + \phi_{ij}^{k'}$ bounded above by $\pi - \epsilon'$, and (c) $K_i\geq\epsilon'$ for any boundary vertex $i$ in $T_Q$. 
	
Consider a pair of triangles $\triangle ijk$ and $\triangle ijk'$ in $(T,l)_S$, and by assumption $\theta_{ij}^k + \theta_{ij}^{k'} \leq \pi - \epsilon$. 	Let $\Theta_{ij}$ be the intersecting angle of the circumcircles of two triangles $\triangle ijk$ and $\triangle ijk'$ on $\mathbb{S}^2$. It is elementary to show that 
	$$
	\Theta_{ij}=\theta_{jk}^i+\theta_{ik}^j+\theta_{jk'}^i+\theta^j_{ik'}
	-(\theta_{ij}^k+\theta_{ij}^{k'})>2\pi-2(\theta_{ij}^k+\theta_{ij}^{k'})\geq2\epsilon.
	$$
	The stereographic projection preserves angles and circles,
	so the intersecting angle of the circumcircles  of $p_N(\triangle ijk)$ and $p_N(\triangle ijk')$ in $T_Q$ is also $\Theta_{ij}$, if $N$ is not contained in $\triangle ijk\cup\triangle ijk'$. 
	Then it is also ordinary to show that this intersecting angle is
	$$
\Theta_{ij}=\phi_{jk}^i+\phi_{ik}^j+\phi_{jk'}^i+\phi^j_{ik'}
	-(\phi_{ij}^k+\phi_{ij}^{k'})=2\pi-2(\phi_{ij}^k+\phi_{ij}^{k'}).
	$$
	Therefore part (b) is true by 
	$$ \phi_{ij}^{k} + \phi_{ij}^{k'}  =\pi-\frac{\Theta_{ij}}{2}\leq \pi - \epsilon.
	$$
	If $i=N$, then the circumcircles of $\triangle ijk$ and $\triangle ijk'$ are mapped to straight lines $p_N(jk)$ and $p_N(jk')$. So the angle between $p_N(jk)$ and $p_N(jk')$, or the inner angle of the polygon $Q$ at $p_N(j)$, is equal to $\pi-\Theta_{ij}\leq\pi-2\epsilon$. So part (c) is true.
	
	\begin{figure}[h]
		\centering\includegraphics[width=0.65\textwidth]{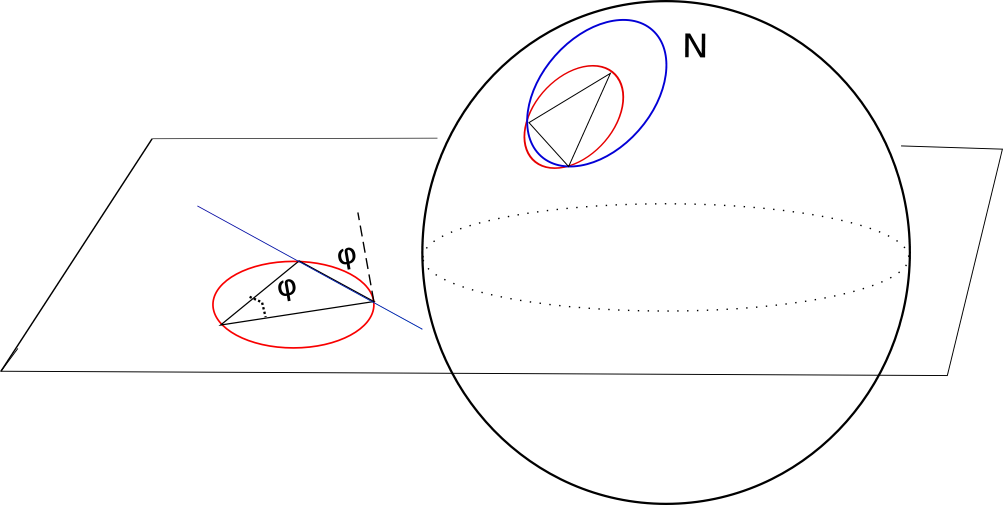}
		\caption{Projection of angles under stereographic projection}
		\label{graph of Delaunay stereographic projection}
	\end{figure}

	Now we prove part (a). Assume $\triangle ijk$ is a triangle in $T$ not containing $N$, and $C$ is the circumcircle of $\triangle ijk$ on $\mathbb S^2$, and $C'$ is the circle on $\mathbb S^2$ containing $\{j,k,N\}$. Then $p_N(C)$ is the circumcircle of $p_N(\triangle ijk)$ and $p_N(C')$ is the straight line $p_N(jk)$, and the intersecting angle of them is equal to the intersecting angle of $C$ and $C'$. It is elementary to show that $\phi^i_{jk}$ is equal to an intersecting angle of $p_N(C)$ and $p_N(C')$, i.e., an intersecting angle of 
	$C$ and $C'$. See Figure \ref{graph of Delaunay stereographic projection} for illustration.  We only need to show that the intersecting angle of $C$ and $C'$ are at least $\epsilon'$ for some constant $\epsilon'(\epsilon)>0$, when $|l|<\delta$ for some constant $\delta(\epsilon)>0$.

		Denote $R$ as the spherical radius of the cirlce $C$.
		Since $(T,l)_S$ is $\epsilon$-regular and $|l|<\delta$, the degree (valence) of any vertex in $T$ is at most $\lfloor2\pi/\epsilon\rfloor$, and $l_{ij},l_{ik},l_{jk}$ are at least $r_1R$ for some constant $r_1=r_1(\epsilon)$.
		Further it is not difficult to show that there exists a constant $r_2(\epsilon)>0$ such that the 1-star neighborhood of $i$ in $T$ contains the open spherical disk $U_i$ centered at $i$ with radius $r_2R$. So $N\notin U_i$. We define $U_j$ and $U_k$ similarly.

	Assume $C_k$ is the circumcircle of the triangle in $T$ that is adjacent to $\triangle ijk$ along the edge $ij$. Then the intersecting angle between $C$ and $C_k$ is $\Theta_{ij}\geq2\epsilon$. Assume $C_k'$ is the circle on $\mathbb S^2$ such that $i,j\in C_k'$ and the intersecting angle between $C$ and $C_k'$ is equal to $2\epsilon$. Denote $D_k$ (resp. $D_k'$, $D$) as the open spherical disk bounded by $C_k$ (resp. $C_k'$, $C$). Then $D_k'\subset D\cup D_k$ and has a diameter less than $r_3R$ for some constant $r_3(\epsilon)>0$. Then $N\notin D_k$ and $N\notin D$ by the convexity of $P$, and so $N\notin D_k'$. Define $D_i'$ and $D_j'$ similarly.

 	Without loss of generality, assume $j'$ is the opposite point of $j$ on $\mathbb S^2$, and $j'\neq N$. Denote $X$ as the tangent plane of $\mathbb S^2$
	 at $j$ in $\mathbb R^3$, and $p'$ as the projection map centered at $j'$ and mapping $\mathbb S^2\backslash\{j'\}$ to $X$. Since $p'$ preserves the angles and disks, $p'(D)$, $p'(D_i')$, $p'(D_j')$, $p'(D_k')$, $p'(U_i)$, $p'(U_j)$, $p'(U_k)$ are all disks on $X$, and the intersecting angle between $p'(D)$ and $p'(D_i)$ (resp. $p
	 (D_j)$, $p'(D_k)$) is $2\epsilon$, and we only need to show that the intersecting angle between $p'(C)$ and $p'(C')$ is at least $\epsilon'$ for some constant $\epsilon'(\epsilon)>0$. The projection $p'$ is very close to an isometry near the point $j$, so
	 if $\delta=\delta(\epsilon)>0$ is sufficiently small, 
	 $$
	 \frac{1}{2}d(x,y)\leq d(p'(x),p'(y))\leq 2d(x,y)
	 $$
	 for any $x,y\in D\cup D_i'\cup D_j'\cup D_k'\cup U_i\cup U_j\cup U_k$. Assume $R'$ is the radius of $p'(D)$ and it is not difficult to show that
	 the radii of $p'(U_i)$ and $p'(U_j)$ and $p'(U_k)$ are at least $r_2R'/4$, and 
	 $ 
	 d(p'(a),p'(b))\geq r_1R'/4
	 $
	 for any two different vertices $a,b$ in $\{i,j,k\}$. 
	 
	 \begin{figure}[h]
		\centering\includegraphics[width=0.4\textwidth]{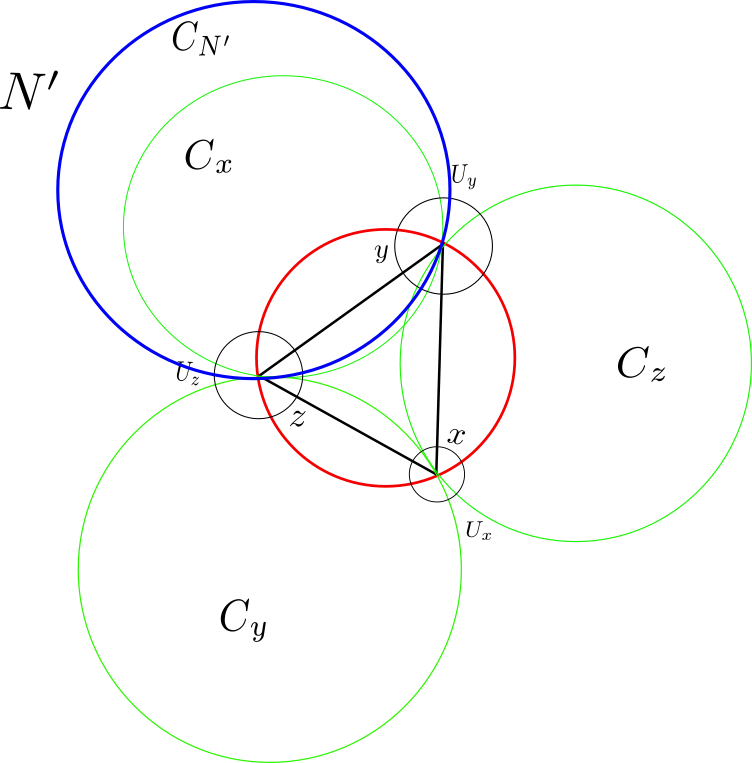}
		\caption{Seven circles.}
		\label{seven circle}
	\end{figure}

	 By a scaling, it suffices to prove the following claim. Assume $\triangle xyz$ is a triangle inscribed in the unit circle $\mathbb S$ in the plane, and all the edge lengths are at least $r_1/4$, and $C_x$ is the circle such that $y,z\in C_x$ and $x$ is not inside $C_x$ and the intersecting angle between $C$ and $C_x$ is $2\epsilon$, and $C_y$ and $C_z$ are defined similarly, and $U_x$ is the open disk centered at $x$ with radius $r_2/4$, and $U_y$ and $U_z$ are defined similarly, and $N'$ is a point that is not strictly inside of the unit circle or $C_x$ or $C_y$ or $C_z$ or $U_x$ or $U_y$ or $U_z$, and $C_{N'}$ is the circle (or the straight line) passing through $y,z,N'$, then the intersecting angle $\theta$ between $C$ and $C_{N'}$ is at least $\epsilon'$ for some constant $\epsilon'(\epsilon)>0$. See Figure \ref{seven circle} for illustration.

	 If the above claim is not true, for some $\epsilon>0$, one can pick a sequence $(x_n,y_n,z_n,N'_n)\in(\mathbb R^2)^4$ such that the resulted intersecting angle $\theta_n$ goes to $0$. By picking a subsequence, we may assume $x_n\rightarrow x\in\mathbb S$ and $y_n\rightarrow y\in\mathbb S$ and $z_n\rightarrow z\in\mathbb S$ and $N'_n\rightarrow N'\in\mathbb R^2\cup\{\infty\}$, and then by the continuity $(x,y,z,N')$ satisfies the conditions in the claim, and the resulted intersecting angle is $\theta=0$. This means that $N'$ is on the unit circle. However this is impossible because by the continuity
	 $$
	 N'\notin D_x\cup D_y\cup D_z\cup U_x\cup U_y\cup U_z\supset\mathbb S
	 $$
	 where $D_x$ (resp. $D_y$, $D_z$) is the open disk bounded by $C_x$ (resp. $C_y$, $C_z$).
	 	\end{proof}

\section{Estimates on Triangulations}
In this section we will introduce some elementary estimates on triangles, and properties of triangulations on a Rimannian surface.
\subsection{Elementary error estimates on triangles}
Denote $|\triangle ABC|$ (resp. $|\triangle A'B'C'|$) as the area of $\triangle ABC$ (resp. $\triangle A'B'C'$).
We have the following estimate on the change of the angles and area of a Euclidean triangle.
      \begin{lemma}[Lemma 3.5 in \cite{wu2020convergence}]
      	\label{comparison lemma for Euclidean triangles}
      	Given a Euclidean triangle $\triangle ABC$ with edge lengths $a$, $b$, and $c$, if all the angles in $\triangle ABC$ are at least $\epsilon>0$, and $\delta<\epsilon^2/48$, and
      	$$
      	|a'-a|\leq \delta a,\quad |b'-b|\leq \delta a,\quad |c'-c|\leq \delta c,
      	$$
      	then
      	$a',b',c'$ form a Euclidean triangle with opposite inner angles $A',B',C'$ respectively such that
      	$$
      	|A'-A|\leq\frac{24}{\epsilon}\delta,
      	$$
      	and
      	$$
      	\bigg||\triangle A'B'C'|-|\triangle ABC|\bigg|\leq\frac{576}{\epsilon^2}\delta\cdot|\triangle ABC|.
      	$$
      \end{lemma}

     By the following lemma we show that the linear map between two Euclidean triangles is close to isometry if their corresponding edge lengths are close.
\begin{lemma}
\label{lip}
	Assume $\triangle ABC$ ($\triangle A'B'C'$) is a Euclidean triangle with edge lengths $a,b,c$ (resp. $a',b',c'$), and all the angles in $\triangle ABC$ are at least $\epsilon>0$, and $\delta<\epsilon^2/576$, and
      	$$
      	|a'-a|\leq \delta a,\quad |b'-b|\leq \delta a,\quad |c'-c|\leq \delta c,
      	$$
and $\lambda_1,\lambda_2$ are the two singular values of the unique linear map sending $\triangle ABC$ to $\triangle A'B'C'$ preserving the correspondence of the vertices. Then 
$$ 1 - \frac{10^4}{\epsilon^4}\delta\leq \lambda_i \leq  1 + \frac{10^4}{\epsilon^4}\delta, \quad i = 1, 2.$$
\end{lemma}
\begin{proof}
By Lemma \ref{comparison lemma for Euclidean triangles}, $|A-A'|,|B-B'|,|C-C'|$ are all less or equal to $24\delta/\epsilon<\epsilon/2$, and thus $A',B',C'$ are all at least $\epsilon/2$. Then again by Lemma \ref{comparison lemma for Euclidean triangles} it is easy to show that 
$$
a'^2\leq\frac{2|\triangle A'B'C'|}{\sin^2(\epsilon/2)}
\leq\frac{64|\triangle ABC|}{\epsilon^2}.
$$
	It is well known that 
	$$
	\lambda_1\lambda_2 = \frac{|\triangle A'B'C'|}{|\triangle ABC|}
	$$ 
	and thus by Lemma \ref{comparison lemma for Euclidean triangles} 
	\begin{equation}
	\label{lambda multiple}
	|\lambda_1\lambda_2-1|<\frac{576}{\epsilon^2}\delta.
	\end{equation}
	In \cite{pinkall1993computing} we can find the formula
	$$\lambda_1^2 + \lambda_2^2 = \frac{a'^2\cot A + b'^2\cot B + c'^2 \cot C}{2|\triangle ABC|}.$$
	Applying this formula to the special case $\triangle A'B'C'=\triangle ABC$, we get
	 $$
	 2=\frac{a'^2\cot A' + b'^2\cot B' + c'^2 \cot C' } {2|\triangle A'B'C'|},
	 $$ 
	 , which implies
	 \begin{equation}
		\label{lambda diff}
\lambda_1^2 + \lambda_2^2 - 2\lambda_1\lambda_2 = \frac{a'^2(\cot A - \cot A') + b'^2(\cot B - \cot B') + c'^2 (\cot C - \cot C')}{2|\triangle ABC|}.
	\end{equation}
	Denote $f(x)=\cot x$, then $f'(x)=-1/\sin^2 x$ and
	$f''(x)=2\cos x/\sin^3 x$.
	By Taylor's expansion, there exists $\xi_A$ between $A$ and $A'$ such that
	$$
	a'^2(\cot A-\cot A')
	=a'^2[f'(A')(A-A')+\frac{1}{2}f''(\xi_A)(A-A')^2]
	$$
	$$=-\frac{a'^2}{\sin^2 A'}(A-A')+\frac{a'^2}{2}f''(\xi_A)(A-A')^2
	=-(2R)^2(A-A')+\frac{a'^2}{2}f''(\xi_A)(A-A')^2
	$$
	where $R$ is the radius of the cicumcircle of $\triangle A'B'C'$, and
	$$
	\left|\frac{a'^2}{2}f''(\xi_A)(A-A')^2\right|
	\leq\frac{64|\triangle ABC|}{\epsilon^2}
\cdot\frac{2}{\sin^3(\epsilon/2)}\cdot\left(\frac{24\delta}{\epsilon}\right)^2
	\leq|\triangle ABC|
	\cdot\frac{10^6\cdot\delta^2}{\epsilon^7}.
	$$
	Combining the similar computation for $B$ and $C$, we get that the right hand side of equation (\ref{lambda diff}) is less or equal to $3\times10^6\delta^2/\epsilon^7$, and thus $|\lambda_1-\lambda_2|\leq
	\sqrt{3\times10^6\delta^2/\epsilon^7}\leq10^4\delta/\epsilon^{4}$. Then by equation (\ref{lambda multiple}) and the fact that $10^4\delta/\epsilon^4\geq576\delta/\epsilon^2$, it is easy to prove that 
$$ 
1 - \frac{10^4}{\epsilon^4}\delta\leq \lambda_i \leq  1 + \frac{10^4}{\epsilon^4}\delta.
$$
\end{proof}

 We also need the following lemma to compare the angle differences between spherical and Euclidean triangles of the same edge lengths. 
     \begin{lemma}[Special case of Lemma 5.3 in \cite{wu2020convergence}]
\label{5.3}
Assume $\triangle ABC$ (resp. $\triangle A'B'C'$) is a Euclidean (resp. spherical) triangle with the edge lengths $a,b,c$, and $diam(\triangle A'B'C')<{\pi}/3$.
Then
$$
|A'-A|\leq  2(a+b+c)^2.
$$
\end{lemma}

\subsection{Geodesic triangulations of surfaces}
 The following cubic estimate quantifies the error between the discrete conformal change and the smooth conformal change. 
    \begin{lemma}[Lemma 4.3 in \cite{wu2020convergence}]
   	\label{cubic estimate}
   	Suppose $(M,g)$ is a  closed Riemannian surface, and $ u\in C^{\infty}(M)$ is a conformal factor. Then there exists $C=C(M,g,u)>0$ such that{} for any $x,y\in M$,
   	$$
   	|d_{e^{2 u}g}(x,y)-e^{\frac{1}{2}(u(x)+ u(y))}d_g(x,y)|\leq Cd_g(x,y)^3.
   	$$
   \end{lemma}


The following lemma shows that a smooth conformal change will preserve the existence of the geodesic triangulation, and its $\epsilon$-regularity.
\begin{lemma}[Part (a) of Lemma 4.4 in \cite{wu2020convergence}]
		\label{geodesic representative and isoperimetry}
Suppose $(M, g)$ is a closed Riemannian surface, and $T$ is a geodesic triangulation
of $(M, g)$. Let $l\in\mathbb{R}^{E(T)}$ denote the geodesic lengths of the edges of $T$, and assume $(T,l)_{E}$ is $\epsilon$-regular.  Given a conformal factor $u\in C^{\infty}(M)$, there exits a constant $\delta=\delta(M,g,u,\epsilon)$ such that if $|l|\leq\delta$ then there exists a geodesic triangulation $T'$ in $(M,e^{2u}g)$ such that $T'$ is homotopic to $T$ relative to $V(T)=V(T')$. 
Further $(T',\bar{l})_E$ is $(\epsilon/2)$-regular, where $\bar{l}\in \mathbb{R}^{E(T')}$ denotes the geodesic lengths of the edges of $T'$ in $(M, e^{2u}g)$.
\end{lemma} 


The following lemma shows that an $\epsilon$-regular geodesic triangulation of a closed surface satisfies the $C$-isoperimetric condition if the edge lengths are sufficiently small. 
 \begin{lemma}[Part (b) of Lemma 4.4 in \cite{wu2020convergence}]
 \label{isoperimetric}
 	Suppose $(M,g)$ is a closed Riemannian surface, and $T$ is a geodesic triangulation of $(M, g)$ with geodesic length $l$ such that $(T, l)_E$ is $\epsilon$-regular. Then there exists a constant $\delta = \delta(M, g, \epsilon)$ such that if  $|l|<\delta$, $(T, l)$ is $C$-isoperimetric for some constant $C = C(M, g, \epsilon)>0$.
 \end{lemma}
 
However, what we really need is the following modified version of Lemma \ref{isoperimetric}.
\begin{lemma}
\label{isoperimetric2}
 	Suppose $(M,g)$ is a closed Riemannian surface, and $T$ is a geodesic triangulation of $(M, g)$ with geodesic length $l$ such that $(T, l)_E$ is $\epsilon$-regular. 
 	Assume $v\in V$ and $star(v)\subset V$ contains $v$ and its neighbors in $T$, and let $\hat{V}=V-star(v)$ and
 	$G=(\hat V,\hat E)$ be the subgraph of $(V(T),E(T))$ generated by $\hat{V}$.
 	Then there exists a constant $\delta = \delta(M, g, \epsilon)$ such that if  $|l|<\delta$, $(\hat T, l|_{\hat E})$ is $C$-isoperimetric for some constant $C = C(M, g, \epsilon)>0$.
\end{lemma}
\begin{proof}
By Lemma \ref{isoperimetric}, we can find constants $\delta(M,g,\epsilon)>0$ and $C(M,g,\epsilon)>0$ such that $(T,l)$ is $C$-isoperimetric if $|l|<\delta$.
Now assume $B=\{i\in\hat V:\exists j\in V-\hat V\text{ s.t. }ij\in E\}$ is the set of  boundary vertices of $G$ in $T$, and $V_0\subset \hat V$, and $\hat\partial V_0$ (resp. $\partial V$) is the boundary of $V_0$ in $G$ (resp. $T$). We consider the following three cases:

\textit{Case 1}: $V_0 \cap B = \emptyset$. Then $ |\hat\partial V_0|_l = |\partial V_0|_l$ and $|V|_l - |V_0|_l \geq  |\hat{V}|_l -|V_0 |_l$. Since $(T,l)$ is $C$-isoperimetric, we have 
$$C |\hat\partial V_0|_l^2 \geq  \min \{ |V_0|_l, |\hat{V}|_l - |V_0|_l\}.$$

\textit{Case 2}: $B \subset V_0$. In this case, $\hat\partial V_0 = \partial (V_0 \cup star(v))$. Since $(T, l)$ is $C$-isoperimetric, we have 
$$C |\hat\partial V_0|_l^2 = C|\partial (V_0\cup star(v))|^2_l \geq \min \{|V_0\cup star(v)|_l, |V|_l - |V_0 \cup star(v)|_l\}$$
Clearly, $ |V_0\cup star(v)|_l \geq |V_0|_l$, and $|V|_l - |V_0 \cup star(v)|_l = |\hat{V}|_l - |V_0|_l$.
Then 
$$C |\hat\partial V_0|_l^2 \geq  \min \{ |V_0|_l, |\hat{V}|_l - |V_0|_l\}.$$

\textit{Case 3}:  $V_0 \cap B \neq \emptyset$ and $ B \not\subset V_0$. It is not difficult to show that $B$ is connected in $V$ since the $1$-skeleton of $T$ is $4$-vertex-connected, so there is an edge $ij\in \hat\partial V_0$ such that $i\in B\cap V_0$ and $j\in B-V_0$. By the $\epsilon$-regularity, the degree of each vertex in $T$ is bounded by $\lfloor2\pi/\epsilon\rfloor$ if $\delta(M,g,\epsilon)$ is sufficiently small, and the ratio of the two edge lengths in a triangle of $T$ is at least $\sin\epsilon$.
So there is a constant $C_1(M,g,\epsilon)>0$ such that $$
C_1l_{ij}\geq\sum_{xy\in E(T)-\hat E}l_{xy}\geq|\partial V_0|_l-|\hat\partial V_0|_l.
$$
Then  
$$C(1+C_1)^2|\hat\partial V_0|_l^2 \geq 
C|\partial V_0|_l^2 \geq 
\min \{ |V_0|_l, |\hat{V}|_l - |V_0|_l\}.$$

\end{proof}

\section{Proof of the Main Theorem \ref{main2}}

\label{proof of the main theorem}
Assume $\epsilon>0$ is a fixed constant and $(T,l)_S$ is $\epsilon$-regular and $|l|<\delta$ where 
$$
\delta=\delta(M,g,X,Y,Z,\epsilon)>0
$$ 
is a sufficiently small constant to be determined.
By Lemma \ref{5.3} we may assume that $(T, l)_E$ is $(\epsilon/2)$-regular.
By Lemma \ref{geodesic representative and isoperimetry}, we may assume that there exists a geodesic triangulation $T'$ of $(M,e^{2\bar u}g)$ such that $T'$ is homotopic to $T$ relative to $V(T)=V(T')$. 
Denote $\bar l\in\mathbb R^{E(T)}\cong\mathbb R^{E(T')}$ as the geodesic edge length of $T'$ in $(M,e^{2\bar u}g)$, and then $(T,\bar l)_S$ is isometric to the unit sphere $(M,e^{2\bar u}g)$ and has zero discrete curvatures. 
Again by Lemma \ref{geodesic representative and isoperimetry} we may assume that $(T,\bar l)_E$ is $(\epsilon/4)$-regular. Then by Lemma \ref{5.3} we may assume $(T,\bar l)_S$ is $(\epsilon/5)$-regular, and thus is strictly Delaunay, and then by Proposition \ref{equivalence} $p(T_P)=\phi(T')$ where $P\in\mathcal P$ is the boundary of the convex hull of $\phi(V(T))$. 
By Lemma \ref{injectivity}, $Q=p_N(\mathring P)$ is a convex polygon, and $T_Q=p_N(\mathring T_P)$ is a geodesic triangulation of $Q$. 
Denote $l_Q\in\mathbb R^{E(T_Q)}$ as the edge lengths in $Q$, and then by Lemma \ref{regular for stereographic condition} there exists a constant $\epsilon'(M,g,X,Y,Z,\epsilon)>0$ such that
$(T_Q,l_Q)_E$ is $\epsilon'$-regular and $K_i\geq\epsilon'$ for any boundary vertex in $(T_Q,l_Q)_E$.
 The combinatorial structures of $T$, $T'$, $T_P$ and $p(T_P)=\phi(T')$ are naturally identified. We also identify the combinatorial structures of $\mathring T_P$ and $T_Q$ and just denote it as $\mathring T$.
Denote $l_P=2\sin(\bar l/2)\in\mathbb R^{E(T)}$ as the edge length of $T_P$ on $P$, and then by equation (\ref{discrete factor for projection}) on $\mathring T$ we have
	\begin{equation*}
	l_Q=w*l_P\quad\text{where}\quad
	w_{i}=\log\frac{2}{|\phi(i)-N|_2^2}=\log\frac{|p_{N}(\phi(i))|_2^2+1}{2}, \quad \forall i\in V(\mathring T).
	\end{equation*}
Denote $l_P'=\bar u*2\sin(l/2)\in\mathbb R^{E(T)}$ and $l_Q'=w*l_P'=(\bar u+w)*2\sin(l/2)\in\mathbb R^{E(\mathring T)}$, and $K(u)\in\mathbb R^{V(\mathring T)}$ as the discrete curvature in $(\mathring T,u*2\sin(l/2))_E$.


In the following proof, for simplicity we will use the notation $a=O(b)$ to represent that if $\delta(M,g,X,Y,Z,\epsilon)$ is sufficiently small, 
then $|a|\leq C b$ for some constant $C(M,g,X,Y,Z,\epsilon)>0$. 
We summarize the remaining part of the proof in three steps: 
 \begin{enumerate}
 	\item [(a)] Estimate the curvature $K(\bar u+w)$ of $(\mathring T,l_Q')_E$ for interior vertices.
 	\item [(b)] Construct a flow $u(t): [0,1]\to \mathbb{R}^{V(\mathring{T})}$, starting from $u(0) = \bar{u}+w$, such that $u(t)$ linearly eliminates the discrete curvature $K(u)$ for interior vertices, i.e.,
 	\begin{equation*}
 		K_i(u(t))=(1-t)K_i(\bar{u}+w)
 	\end{equation*}
 	for any interior vertex $i$ of $\mathring T$.
 	Furthermore, we will also show that $|u'(t)|=O(|l|)$, and $(\mathring{T},u(1)*2\sin (l/2))_E$ is isometric to a convex polygon in the plane.
 	\item [(c)] After a proper normalization, which is a small perturbation, we use the inverse of the stereographic projection to construct the desired polyhedral surface $P\in\mathcal P$.
 \end{enumerate}
 
\subsection{Step 1: The estimate of curvatures}
 By Lemma \ref{cubic estimate}, 
\begin{equation*}
|\bar l_{ij}-(\bar u*l)_{ij}|=O(l_{ij}^3).
\end{equation*}
Notice the fact that $|x-2\sin (x/2)|\leq 10x^3$ if $|x|<0.01$, so
\begin{equation*}
\label{E1}
|(l_P)_{ij}-(l_P')_{ij}|=O(l_{ij}^3),
\end{equation*}
and
\begin{equation}
\label{difference between stereographic metric and virtual stereographic metric}
 \Bigg|\frac{(l_Q')_{ij}-(l_Q)_{ij}}{(l_Q)_{ij}}\Bigg|
= \Bigg|\frac{(l_P')_{ij}-(l_P)_{ij}}{(l_P)_{ij}}\Bigg|= O(l_{ij}^2).
\end{equation}
Given a triangle $\triangle ijk\in F(\mathring T)$, denote $\theta^i_{jk}(u)$ (\textit{resp.} $\bar\theta^i_{jk}$) as the inner angle at $i$ in $\triangle ijk$ in $(\mathring{T}, u*2\sin(l/2))_E$ (\textit{resp.} $(\mathring{T}, {l_Q})_E$),
and $K_i(u)$ as the discrete curvature at $i$ in $\triangle ijk$ in $(\mathring T,u*2\sin(l/2))_E$.

Since $(\mathring T,l_Q)_E$ is $\epsilon'$-regular, by equation (\ref{difference between stereographic metric and virtual stereographic metric}) and Lemma \ref{comparison lemma for Euclidean triangles}, 
\begin{equation*}
\label{E2}
\alpha^i_{jk}:=\bar\theta^i_{jk}-\theta^i_{jk}(\bar{u} + w)=O(l_{ij}^2).
\end{equation*}
So for sufficiently small $\delta(M,g,\epsilon)$, we have
$$
|\alpha^i_{jk}| \leq \frac{\epsilon'}{4}.$$ 
Then
$(\mathring{T},{l_Q'})_E$ is $(\epsilon'/2)$-regular. 
Since $(\mathring{T},{l_Q})_E$ is globally flat, for any $i\in \text{int}( \mathring{T})$,
$$
\sum_{ijk\in {F}}\bar\theta^i_{jk}=2\pi.
$$
So
\begin{equation*}
\label{flow starting point curvature difference}
K_i(\bar u+w)=2\pi-\sum_{ijk\in F}\theta^i_{jk}(\bar u+w)
=\sum_{ijk\in F}(\bar\theta^i_{jk}-\theta^i_{jk}(\bar u+w))
=\sum_{ijk\in F}\alpha^i_{jk} 
\quad \text{if $i\in \text{int}(\mathring{T})$}.
\end{equation*}
Set $x\in\mathbb R^{E(\mathring T)}_A$ be such that
if $ij$ is an interior edge of $\mathring T$
$$
x_{ij}=\frac{\alpha^i_{jk}-\alpha^j_{ik}}{3}+\frac{\alpha^i_{jk'}-\alpha^j_{ik'}}{3},
$$
where $\triangle ijk$ and $\triangle ijk'$ are adjacent triangles in $\mathring T$, and if $ij$ is a boundary edge of $\mathring T$
$$
x_{ij}=\frac{\alpha^i_{jk}-\alpha^j_{ik}}{3},
$$
where $\triangle ijk$ is a triangle in $\mathring T$.
Then 
\begin{equation}
\label{13}
x_{ij}=O(l_{ij}^2),
\end{equation}
and it is straightforward to verify that for any $i\in \text{int}(\mathring T)$
$$
div(x)_i=\sum_{j:j\sim i}x_{ij}
=\sum_{ijk\in F}\alpha_{jk}^i=
K_i(\bar u + w)
$$
using
$\alpha^i_{jk}+\alpha^j_{ik}+\alpha^k_{ij}=0$.

\subsection{Step 2: The construction of the flow}
Consider the sets defined by
$$
\tilde\Omega=\{u\in\mathbb R^{V(\mathring{T})}: \text{$(\mathring{T},u*2\sin\frac{l}{2})_E$  satisfies the triangle inequality and is strictly Delaunay}\},
$$
 and
$$
\Omega=\{u\in\tilde\Omega:  \text{ $(\mathring{T},u*2\sin\frac{l}{2})_E$ is $\frac{\epsilon'}{4}$-regular},|u-(\bar u+w)|\leq1\}.
$$
Notice that $\tilde\Omega$ is an open domain in $\mathbb R^{V(\mathring{T})}$ and $\Omega$ is a compact subset of $\tilde\Omega$. By the construction,  $(\bar{u}+w)$ is in the interior of $\Omega$, since $(\mathring{T}, {l_Q'})_E$ is $(\epsilon'/2)$-regular. 
Given $u\in\tilde\Omega$ and an interior edge $ij$ in $\mathring T$, denote
$$
\eta_{ij}(u)=\frac{1}{2}(\cot\theta^k_{ij}(u)+\cot\theta^{k'}_{ij}(u))
$$
where $\triangle ijk$ and $\triangle ijk'$ are adjacent triangles in $\mathring T$.
Then for $u\in\Omega$,
\begin{equation}
\label{eta}
2\eta_{ij}(u)=\cot\theta^k_{ij}(u)+\cot\theta^{k'}_{ij}(u)
=\frac{\sin(\theta^k_{ij}(u)+\theta^{k'}_{ij}(u))}{\sin\theta^k_{ij}(u)\sin\theta^{k'}_{ij}(u)} 
\geq\sin(\theta^k_{ij}(u)+\theta^{k'}_{ij}(u))
\geq\sin\frac{\epsilon'}{4}
\end{equation}
for any interior edge $ij$ in $\mathring{T}$.

Consider the following system of differential equations on $\tilde\Omega$,
\begin{align}
\label{ode system}
\frac{\partial K_i}{\partial u}\frac{du}{dt}&=-K_i(\bar u+w)=-div(x)_i, \quad\quad&\text{  $i\in \text{int}(\mathring{T})$}, \\
\frac{du_i}{dt}&=(\log 2-2\log (2\sin\frac{ l_{iX}}{2})-\bar u_X)-(\bar u_i+w_i), \quad\quad&\quad{i\in\text{bdy} (\mathring{T})},  \nonumber \\
u(0) &= \bar{u} + w, &\quad \nonumber
\end{align}
where $\bar{u}_X$ is the value of $\bar{u}$ at the marked point $X$ sent to the north pole, and $l_{iX}$ is the length of the edge $iX$ given by $l$. 
We want to show that the solution $u(t)$ exists on $[0,1]$, then it is easy to see that $K_i(u(1))=0$ for an interior vertex $i$ of $\mathring T$, and 
$$
u_i(1)=\log2-2\log(2\sin\frac{l_{iX}}{2})-\bar u_X
$$
for a boundary vertex $i$ of $\mathring T$.

For a boundary vertex $i$ of $\mathring T$, $u_i(t)$ can be easily solved as
$$
u_i(t)=t(\log2-2\log(2\sin\frac{l_{iX}}{2})-\bar u_X)+(1-t)(\bar u_i+w_i),
$$
and
\begin{align*}
\begin{split}
\frac{du_i}{dt}=&(\log 2-2\log (2\sin\frac{ l_{iX}}{2})-\bar u_X)-(\bar u_i+w_i)
\\
=&\log 2-2\log (2\sin\frac{ l_{iX}}{2})-\bar u_X
-
\bar u_i-\log\frac{2}{(l_P)_{iX}^2}
\\
=&-{2\log (2\sin\frac{l_{iX}}{2})}
-\bar u_X-\bar u_i+2\log{(l_P)_{iX}}
\\
=&-2\log (l_P')_{iX}+2\log (l_P)_{iX}\\
=&O(l_{iX}^2).
\end{split}
\end{align*}


Now let us focus on solving $u_i(t)$ for all the interior vertices of $\mathring T$. Let $\hat{V}$ be the set of  interior vertices of $\mathring T$, and $G=(\hat V,\hat E)$ be the subgraph of $(V(\mathring T),E(\mathring T))$ generated by $\hat V$. It is easy to show that $G$ is nonempty and connected. Let $\hat u\in\mathbb R^{\hat V}$ and $\hat x\in\mathbb R^{\hat E}_A$ and $\hat \eta\in\mathbb R^{\hat E}$
be the restrictions of $u$ and $x$ and $\eta$ respectively on $G=(\hat V,\hat E)$, and $\hat\Delta=\Delta_{\hat\eta}$ be the associated discrete Laplacian on $G$. 
Then by Proposition \ref{curv diff} it is straightforward to verify that 
Equation (\ref{ode system}) can be rewritten as 
\begin{equation}
\label{ode2}
(D-\hat\Delta)\frac{d\hat u}{dt}=-div(\hat x)+y,
\end{equation}
where 

(a) $D\in\mathbb R^{\hat V\times\hat V}$ is a nonzero diagonal matrix and 
$$
D_{ii}=\sum_{j\sim i:j\notin\hat V}\eta_{ij}\geq0,
$$
and

(b)
$$
y_i=\sum_{j\sim i:j\notin \hat V}\eta_{ij}\frac{du_j}{dt}-
\sum_{j\sim i:j\notin \hat V}x_{ij}.
$$
For $u\in\tilde\Omega$, it is easy to show that $(D-\hat\Delta)$ is positive definite, by the fact that $G$ is connected, $\eta_{ij}>0$ for any $ij\in\hat E$, and $D$ is nonzero and non-negative. So equation (\ref{ode2}) locally has a unique solution in $\tilde\Omega$.

Assume the maximum existing open interval for the solution $\hat u(t)\in\Omega$ is $[0,T_0)$ where $0< T_0 \leq +\infty$. 
For $t\in[0,T_0)$, we have
$$
\frac{d\hat u}{dt}(t)=O(|l|\cdot|\hat V|_l^{1/2})
$$
by Lemma \ref{estimate for discrete elliptic operator}, and Lemma \ref{isoperimetric2}, and equation (\ref{eta}) and (\ref{13}), and the fact that

$$
y_i=
\sum_{j\sim i:j\notin \hat V}(\eta_{ij}\frac{du_j}{dt}-
x_{ij})
=O\left(
\sum_{j\sim i:j\notin \hat V}(\eta_{ij} l_{jX}^2+
l_{ij}^2)
\right)
=O(D_{ii}|l|^2)=O(D_{ii}|l|\cdot|\hat V|_l^{1/2}).
$$
Further
\begin{align*}
\begin{split}
&|\hat V|_l\leq|V|_l=\sum_{ij\in E}l_{ij}^2= O(\sum_{ij\in E}\bar l_{ij}^2)
=O(\sum_{ijk\in F}(\bar l_{ij}^2+\bar l_{jk}^2+\bar l_{ik}^2))\\
= &O(\sum_{ijk\in F} Area(\triangle ijk,\bar l)_S)
=O(Area((T,\bar l)_S))=O(Area(\mathbb S^2))=O(1),
\end{split}
\end{align*}
and thus $(du/dt)(t)=O(|l|)$ for $t\in[0,T_0)$.

If $T_0 \leq 1$, combining Lemma \ref{comparison lemma for Euclidean triangles}, we have 
$$
|u(T_0) - (\bar{u} + w)| = O(|l|), \quad \text{and } \quad |\theta_{jk}^i(u(T_0)) - \theta_{jk}^i(\bar{u}+w)| = O(|l|).
$$
This implies that $u(T_0)\in \text{int}(\Omega)$ if $\delta$ is sufficiently small, which contradicts to the maximality of $T_0$. Thus, $T_0> 1$ and $u(1)$ is well-defined. Further we have that

(a) $K_i(u(1))=0$ for any interior vertex $i$ of $\mathring T$, and

(b) $u_i(1)=\log2-2\log(2\sin\frac{l_{iX}}{2})-\bar u_X$ for any boundary vertex $i$ of $\mathring T$, and

(c) $u(1)-(\bar u+w)=O(|l|)$, and

(d) $(\mathring T,u(1)*2\sin\frac{l}{2})_E$ is strictly Delaunay, and 

(e) $K_i(u(1))>0$ for any boundary vertex $i$ in $\mathring T$.

\subsection{Step 3: The normalization and the inverse of the stereographic projection} 
We know that
$(\mathring{T},u(1)*2\sin \frac{l}{2})_E$ is isometric to a closed convex polygon in $\mathbb C$. Let $f$ be the piecewise linear map from $(\mathring{T},u(1)*2\sin \frac{l}{2})_{E}$ to $(\mathring{T},l_Q)_E$ that preserves the triangulation and is linear on each triangle. From equation (\ref{difference between stereographic metric and virtual stereographic metric}) and the fact that $u(1)-(\bar u+w)=O(|l|)$, we can deduce that 
$$
\Big|\frac{(u(1)*2\sin\frac{l}{2})_{ij}-({l_Q})_{ij}}{({l_Q})_{ij}}\Big| = \Big|\frac{(u(1)*2\sin\frac{l}{2})_{ij}-  ((\bar{u}+w)*2\sin\frac{l}{2})_{ij} }{({l_Q})_{ij}}\Big|   + O(|l|^2)= O(|l|).
$$
Then by Lemma \ref{lip}, $\|Df\|_2$ and $\|Df^{-1}\|_2$ are both $(1+C|l|)$-Lipschitz for some constant $C(M,g,X,Y,Z,\epsilon)>0$. So the distance $d_{YZ}$ between $Y$ and $Z$ in $(\mathring{T},u(1)*2\sin (l/2))_E$ lies in $[1-C|l|,1+C|l|]$.
So we can scale $(\mathring{T},u(1)*2\sin (l/2))_E$ by letting $\tilde u=u(1)-\log d_{YZ}$, and then $(\mathring{T},\tilde u*2\sin (l/2))_E$ is still isometric to a convex polygon and the distance between $Y$ and $Z$ is $1$, and 
$$
\Big|\frac{\tilde u*2\sin\frac{l}{2}-{(l_Q)}_{ij}}{{(l_Q)}_{ij}}\Big|= \Big|\frac{\tilde u*2\sin\frac{l}{2}-u(1)*2\sin\frac{l}{2}}{{(l_Q)}_{ij}}\Big| + O(|l|) = O(|l|).
$$
Let $g$ be the isometry from $(\mathring{T},\tilde u*2\sin\frac{l}{2})_E$ to a closed convex polygon $Q_1$  in $\mathbb C$ such that $g(Z)=0$ and $g(Y)=1$. Then for any $i\in \mathring{V}$, the above bi-Lipschitz property of $f$ implies that 
$$
\left|\log\frac{|g(i)|_2}{|p_N(\phi(i))|_2}\right| = O(|l|).
$$

Now we are ready to project the points in the plane back to the sphere. Let 
$$
{V_1}=(p_{N}|_{\mathbb S^2})^{-1}(g({V(\mathring T)}))\cup\{N\}
$$ 
and ${P_1}$ be the convex hull of ${V_1}$. 
Then by part (b) of Lemma \ref{injectivity},
$P_1\in\mathcal P$ and
$p_N(\mathring P_1)=Q_1$ and $p_N(\mathring T_{ P_1})=g(\mathring T)$. Naturally identify the combinatorial structures of $T$ and $T_{P_1}$, and denote $l_{P_1}\in\mathbb R^{E(T)}$ as the edge length on $P_1$. We will verify that 
$$
l_{P_1}=u*2\sin\frac{l}{2}
$$
where $u_X=\bar u_X+\log d_{YZ}$ and 
$$
u_i=\tilde u_i-w_i',\quad\text{where}\quad w_i'=\log\frac{|g(i)|^2+1}{2}
$$
if $i\in V(\mathring T)$.
If $ij\in E(\mathring T)$, 
$$
(l_{P_1})_{ij}=(u*2\sin\frac{l}{2})_{ij}
$$
by implementing Lemma \ref{injectivity} on $l_{P_1}$ and $l_{Q_1}$. For edge $iX\in E(T)$, we have that
\begin{align*}
\begin{split}
&\log(u*2\sin\frac{l}{2})_{iX} \\
=& \log(2\sin \frac{l_{iX}}{2}) +{\frac{1}{2}(\bar{u}_X + \log d_{YZ} + \log 2-2\log (2\sin\frac{ l_{iX}}{2})-\bar{u}_X - \log d_{YZ}}-w_i')
\\
=&  {\frac{1}{2}(\log 2 - w_i'}) 
= \frac{1}{2}\log\frac{4}{|g(i)|^2+1}
=\frac{1}{2}\log (l_{P_1})_{iX}^2=\log (l_{P_1})_{iX}.
\end{split}
\end{align*}
So $u$ is our desired discrete conformal factor. As we mentioned in Remark \ref{unique}, such $u$ is known to be unique.
It remains to show $u_i-\bar u_i=O(|l|)$ for any $i\in V$. Notice that

$$
|w_i'-w_i|=\Big|\log\frac{|g(i)|_2^2+1}{2}-\log\frac{|p_N(\phi(i))|_2^2+1}{2} \Big|=\Big|\log\frac{|g(i)|_2^2+1}{|p_N(\phi(i)|_2^2+1}\Big| = O(|l|).
$$
So restricted on $V(\mathring T)$, we have that
$$
u-\bar u
=(\tilde u-w')-\bar u
=u(1)-\log d_{YZ}-w'-\bar u=(u(1)-\bar u-w)+(w-w')-\log d_{YZ}=O(|l|).
$$
On vertex $X$ we have that
$u_X-\bar u_X=\log d_{YZ}=O(|l|)$.

\bibliographystyle{unsrt}
\bibliography{main}

\end{document}